\documentclass[10pt]{article}
\usepackage[utf8]{inputenc}
\usepackage[T1]{fontenc}
\usepackage{amsmath}
\usepackage{amsfonts}
\usepackage{amssymb}
\usepackage[version=4]{mhchem}
\usepackage{stmaryrd}
\usepackage{bbold}
\usepackage{hyperref}
\hypersetup{colorlinks=true, linkcolor=blue, filecolor=magenta, urlcolor=cyan,}
\usepackage{amsthm}

\newtheorem{theorem}{Theorem}[section]
\newtheorem{lemma}[theorem]{Lemma}
\newtheorem{proposition}[theorem]{Proposition}
\newtheorem{corollary}[theorem]{Corollary}

\newtheorem{definition}{Definition}[section]

\theoremstyle{remark}

\title{Homogeneous Groups and Covers }

\author{Eric Primozic}
\date{April 2014}

\begin{document}
\maketitle

\section*{Introduction}
Given an $n$-dimensional vector space $V$ over a field $F$ and ordered bases $$\left(x_{1}, \ldots, x_{n}\right),\left(y_{1}, \ldots, y_{n}\right)$$ of $V$, there exists a unique isomorphism $f$ of $V$ such that $f\left(x_{i}\right)=y_{i}$ for all $i \leq n$. It is natural to ask how this property generalizes to generating sequences of a finite group with vector space isomorphisms replaced by group isomorphisms and ordered bases replaced by irredundant generating sequences of a particular length.

\begin{definition} We say that a sequence of elements $\left(x_{1}, \ldots, x_{n}\right)$ in a finite group $G$ is a generating sequence of $G$ if $\left\langle x_{1}, \ldots, x_{n}\right\rangle=G$.
\end{definition}

\begin{definition} A generating sequence $\left(x_{1}, \ldots, x_{n}\right)$ of a finite group $G$ is called irredundant if no proper subsequence of $\left(x_{1}, \ldots, x_{n}\right)$ also generates $G$.
\end{definition}

It is clear that the vector space property from the previous paragraph does not generalize to all finite groups. For example, $G=Z_{2} \times Z_{3}$, where $Z_{n}$ denotes the cyclic group of order $n$ written multiplicatively, does not satisfy this property since $\langle(a, 1),(1, b)\rangle=G$, where $\langle a\rangle=Z_{2},\langle b\rangle=Z_{3}$, and there is no isomorphism of $G$ mapping $a$ to $b$.

Finite groups satisfying the above-mentioned vector space property for generating sequences of length $n$ are called homogeneous groups of rank $n$. These groups were first studied by Gaschütz and the Neumanns. For every finite group $G$ that can be generated by $n$ elements, one can associate a homogeneous group $H(n, G)$, called the homogeneous cover of $G$ of rank $n$, having nice properties in relation to $G$ and satisfying certain universal mapping properties. It is hoped that homogeneous covers of finite groups can be used to develop a $K$-theory for finite groups.

In this thesis, we give a survey of results on homogeneous groups and covers. From previously made computations of specific homogeneous covers, we discuss more general results and theorems that arise from these computations.

\section{Basic Properties of Generating Sequences}
The primary objects in our study will be generating sequences of finite groups.

\begin{definition} For a finite group $G$, let $\Gamma_{n}(G)$ denote the set of all length $n$ generating sequences of $G$.
\end{definition}

\begin{definition} For a finite group $G$, let $r(G)$ denote the smallest integer $n$ for which $\Gamma_{n}(G) \neq \emptyset$.
\end{definition}

From our discussion of the vector space property in the introduction, it is important to study how the automorphisms of a finite group $G$ act on $\Gamma_{n}(G)$ for all positive integers $n$. Such an action should indicate how a group might not satisfy the vector space property. We define the action of $\operatorname{Aut}(G)$ on $\Gamma_{n}(G)$ by $f\left(x_{1}, \ldots, x_{n}\right)=\left(f\left(x_{1}\right), \ldots, f\left(x_{n}\right)\right)$ for $f \in \operatorname{Aut}(G)$ and $\left(x_{1}, \ldots, x_{n}\right) \in \Gamma_{n}(G)$. From the group structure of $\operatorname{Aut}(G)$, it is clear that the required properties for these relations to define an action are satisfied. The vector space property for a particular integer $n$ is then equivalent to the action of $\operatorname{Aut}(G)$ on $\Gamma_{n}(G)$ being transitive. For the next proposition, we assume $G$ can be generated by $n$ elements.

\begin{proposition} Aut $(G)$ acts freely on $\Gamma_{n}(G)$.
\end{proposition}

\begin{corollary} With respect to the action of $\operatorname{Aut}(G), \Gamma_{n}(G)$ is the disjoint union of $\left|\Gamma_{n}(G)\right| /|\operatorname{Aut}(G)|$ orbits. Each orbit contains $|\operatorname{Aut}(G)|$ elements.
\end{corollary}

\begin{definition} Let $h_{n}(G)=\left|\Gamma_{n}(G)\right| /|\operatorname{Aut}(G)|$.
\end{definition}

\begin{proposition} A finite group $G$ is homogeneous of rank $n$ if and only if $h_{n}(G)=1$.
\end{proposition}

Of course, it is true that a homogeneous group $G$ of $\operatorname{rank} n$ indeed satisfies $r(G)=n$.

\begin{proposition} Let $G$ be a nontrivial finite group. If $G$ is a homogeneous group of rank $n$, then $r(G)=n$.
\end{proposition}

\begin{proof} Suppose that $n>r(G)$. There exists $x=\left(x_{1}, \ldots, x_{n}\right) \in \Gamma_{n}(G)$ and $y=\left(y_{1}, \ldots, y_{r(G)}\right) \in \Gamma_{r(G)}(G)$. We may assume that $x_{n} \neq 1$. Note that $w=\left(y_{1}, \ldots, y_{r(G)}, z\right) \in \Gamma_{n}(G)$ where $z=(1, \ldots, 1)$ is the length $n-r(G)$ element of $G^{n-r(G)}$ consisting of ones. As $G$ is homogeneous of rank $n$, there exists $f \in \operatorname{Aut}(G)$ such that $f(x)=w$ which is impossible since such an automorphism $f$ would map the nontrivial element $x_{n}$ to 1 .
\end{proof}

For any prime $p, h_{n}\left(Z_{p}^{n}\right)=1$ since any group automorphism of $Z_{p}^{n}$ is also a vector space isomorphism. $h_{n}(G)$ can be computed more generally using GAP [3]. For the example in the introduction, $h_{2}\left(Z_{2} \times Z_{3}\right)=12 \neq 1$ as one would expect since $Z_{2} \times Z_{3}$ does not satisfy the vector space property described in the introduction. However, we note that $h_{2}\left(Z_{2}^{2} \times Z_{3}^{2}\right)=1$. We will later establish general principles that allow us to compute $h_{n}(G)$ for an infinite number of finite groups $G$ and certain integers $n$.

\section{Free Groups and Generating Sequences}
Throughout this thesis, we let $F_{n}$ denote the free group on $n$ generators $a_{1}, \ldots, a_{n}$. For a finite group $G$ and $s=\left(x_{1}, \ldots, x_{n}\right) \in \Gamma_{n}(G)$, there exists a homomorphism $f: F_{n} \rightarrow G$ defined by $f\left(a_{i}\right)=x_{i}$ for $i \leq n$.

\begin{definition} Let $G$ be a finite group such that $n \geq r(G)$. Given $s=$ $\left(x_{1}, \ldots, x_{n}\right) \in \Gamma_{n}(G)$, we let $\pi_{s}: F_{n} \rightarrow G$ be the unique homomorphism from $F_{n}$ to $G$ such that $\pi_{s}\left(a_{i}\right)=x_{i}$ for $i \leq n$. We also let $K_{s}=\operatorname{ker} \pi_{s}$ for $s \in \Gamma_{n}(G)$.
\end{definition}

As the following proposition shows, the kernels of surjective homomorphisms from $F_{n}$ to a finite group $G$ are closely related to the action of $\operatorname{Aut}(G)$ on $\Gamma_{n}(G)$.

\begin{proposition} Let $G$ be a finite group such that $n \geq r(G)$. Let $s_{1}, s_{2} \in$ $\Gamma_{n}(G)$. Then $s_{1}$ and $s_{2}$ are equivalent under the action of $\operatorname{Aut}(G)$ on $\Gamma_{n}(G)$ if and only if $K_{s_{1}}=K_{s_{2}}$.
\end{proposition}

\begin{proof} Suppose $s_{1}$ and $s_{2}$ are equivalent. Then there exists $f \in \operatorname{Aut}(G)$ such that $\pi_{s_{2}}=f \circ \pi_{s_{1}}$. It follows that $K_{s_{1}}=K_{s_{2}}$ since $f$ is an automorphism.

Conversely, suppose $K_{s_{1}}=K_{s_{2}}$. For $i=1$ and $i=2$, there exists $\overline{\pi_{s_{i}}}: F_{n} / K_{1} \rightarrow G$ such that $\pi_{s_{i}}=\overline{\pi_{s_{i}}} \circ \pi$ where $\pi: F_{n} \rightarrow F_{n} / K_{1}$ is the quotient map. Then $h=\overline{\pi_{s_{2}}} \circ \overline{\pi_{s_{1}}}-1$ is an automorphism of $G$ satisfying $h\left(s_{1}\right)=s_{2}$.
\end{proof}

\begin{corollary} There is a one-to-one correspondence between orbits of $\Gamma_{n}(G)$ under the action of $\operatorname{Aut}(G)$ and kernels of surjective homomorphisms $F_{n} \rightarrow$ $G$.
\end{corollary}

Proposition 2.1 is particularly useful in describing finite subdirect products of groups. The following proposition will be used in our discussion of homogeneous covers.

\begin{proposition} Let $n \in \mathbb{N}$ and let $G_{1}, \ldots, G_{m}$ be finite groups that can be generated by $n$ elements. Then any subdirect product $G \leq G_{1} \times \cdots \times G_{m}$, for which $r(G) \leq n$, is isomorphic to $F_{n} / K$ where $K$ can be written as an intersection of kernels of surjective homomorphisms $f_{i}: F_{n} \rightarrow G_{i}$ for $i \leq n$.
\end{proposition}

\begin{proof} As $G$ can be generated by $n$ elements, there exists a surjective homomorphism $f: F_{n} \rightarrow G$. Let $p_{i}: G \rightarrow G_{i}$ be the ith coordinate map for $i \leq n$. Then $f_{i}=p_{i} \circ f$ is surjective for $i \leq n$. We then observe that $f(x)=\left(f_{1}(x), \ldots, f_{m}(x)\right)$ for $x \in F_{n}$. It follows that $\operatorname{ker} f=\bigcap_{i} \operatorname{ker} f_{i}$.
\end{proof}
Here, we state a theorem of Gaschütz [4] that will be needed for later proofs.

\begin{theorem}[Gaschütz's Lemma] Let $G$ be a finite group and let $f$ : $G \rightarrow H$ be a surjective homomorphism. Suppose $n \geq r(G)$. Then for $s=$ $\left(s_{1}, \ldots, s_{n}\right) \in \Gamma_{n}(H)$, there exists $t=\left(t_{1}, \ldots, t_{n}\right) \in \Gamma_{n}(G)$ such that $f\left(t_{i}\right)=$ $s_{i}$ for $i \leq n$.
\end{theorem}

\section{Properties of Homogeneous Covers}
Assume $G$ is a finite group and $r(G) \leq n$. Let $O_{1}, \ldots, O_{h_{n}(g)}$ be the orbits of $\operatorname{Aut}(G)$ on $\Gamma_{n}(G)$. For $i \leq h_{n}(G)$, we let $s_{i} \in O_{i}$. We take $\pi_{s_{i}}: F_{n} \rightarrow G$ and $K_{s_{i}} \leq F_{n}$ for $i \leq h_{n}(G)$ to be defined as in Definition 2.1. From Corollary 2.2, $K_{1}, \ldots, K_{h_{n}(G)}$ are the distinct kernels of all surjective homomorphisms $F_{n} \rightarrow G$. Let $K=\bigcap_{i \leq h_{n}(G)} K_{i} . K$ is the kernel of the map $F_{n} \rightarrow G^{h_{n}(G)}$ defined by $a_{i} \rightarrow\left(\pi_{s_{1}}\left(a_{i}\right), \ldots, \pi_{s_{h_{n}(G)}}\left(a_{i}\right)\right)$ for $i \leq n$. Hence, $F_{n} / K$ is a subdirect product of $h_{n}(G)$ copies of $G$.

\begin{definition} We define $H(n, G)=F_{n} / K$ to be the $n t h$ homogeneous cover of $G$.
\end{definition}

\begin{proposition} If $G$ is homogeneous of rank $n, G=H(n, G)$.
\end{proposition}

\begin{proof} All surjective homomorphisms $F_{n} \rightarrow G$ have the same kernel $K_{0}$. Hence, $G \approx F_{n} / K_{0}=H(n, G)$.
\end{proof}
\section{Subdirect Products and Homogeneous Covers}
Let $G$ be an $n$-generated finite group. Propositions 2.1 and 2.3 imply that any $n$-generated subdirect product of copies of $G$ is a quotient of $H(n, G)$. Hence, $H(n, G)$ is the largest subdirect product of copies of $G$ that can be generated by $n$ elements. These remarks prove the following.

\begin{proposition} Any n-generated subdirect product of copies of $G$ is a quotient of $H(n, G)$.
\end{proposition}

\begin{corollary} Any n-generated subdirect product of copies of $H(n, G)$ must be isomorphic $H(n, G)$.
\end{corollary}

\begin{proof} Observe that any subdirect product of copies of $H(n, G)$ is a subdirect product of copies of $G$.
\end{proof}
\begin{corollary} Any n-generated subdirect product of quotients of $G$ must be a quotient of $H(n, G)$.
\end{corollary}

\begin{proof} Let $H \leq G / H_{1} \times \cdots \times G / H_{m}$ be an $n$-generated subdirect product of $G / H_{1}, \ldots, G / H_{m} . \quad H=\left\langle\left(s_{11} H_{1}, \ldots, s_{m 1} H_{m}\right), \ldots,\left(s_{1 n} H_{1}, \ldots, s_{m n} H_{m}\right)\right\rangle$ where $\left(s_{i 1} H_{i}, \ldots, s_{i n} H_{i}\right)$ is a generating sequence of $G / H_{i}$ for $i \leq m$. From Gaschütz's Lemma, there exists a lift of each $\left(s_{i 1} H_{i}, \ldots, s_{i n} H_{i}\right)$ to a generating sequence $\left(t_{i 1}, \ldots, t_{i n}\right)$ of $G$ for $i \leq m$ via the quotient map. Then $$K=\left\langle\left(t_{11}, \ldots, t_{m 1}\right), \ldots,\left(t_{1 n}, \ldots, t_{m n}\right)\right\rangle \leq G^{m}$$ is an $n$-generated subdirect product of $m$ copies of $G$. Notice also that $K$ maps surjectively onto $H$ by the quotient map $G^{m} \rightarrow G / H_{1} \times \cdots \times G / H_{m} . K$ is a quotient of $H(n, G)$ by Proposition 4.1. Hence, $H$ is a quotient of $H(n, G)$.
\end{proof}

Proposition 4.1 implies the following proposition.

\begin{proposition} Let $G$ be a finite n-generated group. Suppose $n>m \geq$ $r(G)$. Then there exists a surjection $H(n, G) \rightarrow H(m, G)$.
\end{proposition}
\begin{proof} The proposition is proven by noting that $H(m, G)$ is an $n$-generated subdirect product of copies of $G$. Hence, Proposition 4.1 applies. We give an explicit example of the surjection.

As at the beginning of the section, let $s_{1}, \ldots, s_{h_{n}(G)}$ be a set of representatives from the orbits of $\operatorname{Aut}(G)$ on $\Gamma_{n}(G)$. After reordering the $s_{i}$ if necessary, we can assume that for some $l \leq h_{n}(G), s_{i}=\left(x_{i}, z\right)$ for all $i \leq l$ where $z \in G^{h_{n}(G)-m}$ consists entirely of ones. Then $l=h_{m}(G)$ and $x_{1}, \ldots, x_{h_{m}(G)}$ is a set of representatives from the orbits of $\operatorname{Aut}(G)$ on $\Gamma_{m}(G)$. We assume that $H(n, G) \leq G^{h_{n}(G)}$ is of the form described at the beginning of this section where $s_{i}$ corresponds to the $i t h$ factor of $G^{h_{n}(G)}$ for $i \leq h_{n}(G)$. We next let $p: G^{h_{n}(G)} \rightarrow G^{h_{m}(G)}$ be the projection map onto the first $h_{m}(G)$ coordinates of $G^{h_{n}(G)}$. Then $p$ restricted to $H(n, G)$ maps onto $H(m, G)$.
\end{proof}

It is an open question of when the homomorphism from the previous proposition splits. In our later computations of homogeneous covers, we will see that the homomorphism splits for homogeneous covers of simple groups and abelian groups. The splitting of the homomorphism from the proposition
is related to developing a $K$-theory for finite groups as mentioned in the introduction.

We next prove that $H(n, G)$ is homogeneous of rank $n$. First, we recall Goursat's lemma.

\begin{lemma}[Goursat] Let $H_{1}$ and $H_{2}$ be finite groups. Suppose $H \leq$ $H_{1} \times H_{2}$. There exists subgroups $S_{1} \unlhd \overline{S_{1}} \leq H_{1}, S_{2} \unlhd \overline{S_{2}} \leq H_{2}$ and an isomorphism $f: \overline{S_{1}} / S_{1} \rightarrow \overline{S_{2}} / S_{2}$ such that $H=\left\{\left(g_{1}, g_{2}\right) \in \overline{S_{1}} \times \overline{S_{2}}\right.$ : $\left.f\left(g_{1} S_{1}\right) S_{2}=g_{2} S_{2}\right\}$. The converse is also true.
\end{lemma}

\begin{proposition} $H(n, G)$ is homogeneous of rank $n$.
\end{proposition}
\begin{proof} Suppose there exist generating sequences $s=\left(s_{1}, \ldots, s_{n}\right)$ and $t=$ $\left(t_{1}, \ldots, t_{n}\right)$ that are not equivalent under the action of $\operatorname{Aut}(H(n, G))$. Consider $H=\left\langle\left(s_{1}, t_{1}\right), \ldots,\left(s_{n}, t_{n}\right)\right\rangle \leq H(n, G) \times H(n, G) . \quad H$ is a subdirect product of 2 copies of $H(n, G)$. From Goursat's lemma, there exist $H_{1}, H_{2} \unlhd$ $H(n, G)$ and an isomorphism $f: H(n, G) / H_{1} \rightarrow H(n, G) / H_{2}$ such that $$H=\left\{\left(g_{1}, g_{2}\right) \in H(n, G) \times H(n, G): f\left(g_{1} H_{1}\right) H_{2}=g_{2} H_{2}\right\}.$$ As $s$ and $t$ are not equivalent under the action of $\operatorname{Aut}(H(n, G)), H_{1}, H_{2} \neq\{1\}$. It follows that $|H|>|H(n, G)|$, contradicting Corollary 4.2.
\end{proof}

Before describing the universal mapping properties of $H(n, G)$, we note that $H(n, G)$ and $G$ have similar properties as a result of $H(n, G)$ being a subdirect product of copies of $G$.

\begin{proposition}

\begin{enumerate}
 
  \item $\exp (H(n, G))=\exp (G)$.
  \item $H(n, G)$ is abelian if and only if $G$ is abelian.

  \item $H(n, G)$ is nilpotent if and only if $G$ is nilpotent.

  \item $H(n, G)$ and $G$ have the same Jordan-Hölder composition factors but not necessarily up to multiplicity.

  \item $H(n, G)$ is solvable if and only if $G$ is solvable.

\end{enumerate}
\end{proposition}

\section{Universal Mapping Properties}
Let $G$ be a finite $n$-generated group. From the previous section, we saw that $H(n, G)$ was the largest $n$-generated subdirect product of copies of $G$. In this section, we shall prove that $H(n, G)$ satisfies several universal mapping properties. One of these universal mapping properties classifies $H(n, G)$ as the smallest object with respect to a certain universal mapping property.

\begin{proposition} Let $G$ be a finite homogeneous group of rank n. Suppose there exists a surjective homomorphism $f: G \rightarrow H$ where $H$ is some finite group. For any $s=\left(s_{1}, \ldots, s_{n}\right) \in \Gamma_{n}(H)$ and $t=\left(t_{1}, \ldots, t_{n}\right) \in \Gamma_{n}(G)$, there exists a surjective homomorphism $h: G \rightarrow H$ such that $h\left(t_{i}\right)=s_{i}$ for $i \leq n$.
\end{proposition}
\begin{proof} Let $s=\left(s_{1}, \ldots, s_{n}\right) \in \Gamma_{n}(H)$ and $t=\left(t_{1}, \ldots, t_{n}\right) \in \Gamma_{n}(G)$. From Gaschütz's lemma, $s$ lifts via $f$ to a generating sequence $q=\left(q_{1}, \ldots, q_{n}\right)$ of $G$. By homogeneity of $G$, there exists an automorphism $k: G \rightarrow G$ such that $k\left(t_{i}\right)=q_{i}$ for $i \leq n$. Then $h=f \circ k$ satisfies $h\left(t_{i}\right)=s_{i}$ for $i \leq n$.
\end{proof}
\begin{proposition} Let $G$ be an n-generated finite group. Pick any surjective homomorphism $\pi: H(n, G) \rightarrow G$. For any finite n-generated homogeneous group $H$ and surjective homomorphism $j: H \rightarrow G$, there exists a surjective homormorphism $f: H \rightarrow H(n, G)$ such that $j=\pi \circ f$.
\end{proposition}
\begin{proof} Let $s_{1}=\left(s_{11}, \ldots, s_{1 n}\right), \ldots, s_{h_{n}(G)}=\left(s_{h_{n}(G) 1}, \ldots, s_{h_{n}(G) n}\right)$ be a set of representatives from the orbits of $\operatorname{Aut}(G)$ on $\Gamma_{n}(G)$. From Gaschütz's lemma, each $s_{i}$ lifts through $j$ to a generating sequence $t_{i}=\left(t_{i 1}, \ldots, t_{i n}\right)$ of $H$ for $i \leq h_{n}(G)$. Consider $H_{0}=\left\langle\left(t_{11}, \ldots, t_{h_{n}(G) 1}\right), \ldots,\left(t_{1 n}, \ldots, t_{h_{n}(G) n}\right)\right\rangle \leq H^{h_{n}(G)}$. $H_{0}$ is an $n$-generated subdirect product of $h_{n}(G)$ copies of $H$. Then $H_{0} \approx H$ by Corollary 4.2. Then $f_{0}: H_{0} \rightarrow H(n, G)$ defined by $f_{0}\left(g_{1}, \ldots, g_{h_{n}(G)}\right)=$ $\left(j\left(g_{1}\right), \ldots, j\left(g_{h_{n}(G)}\right)\right)$ is a surjective homomorphism. Taking $H=H_{0}$, we view $f_{0}$ as having domain $H$.

We can now apply the previous proposition to complete the proof. Let $\left(g_{1}, \ldots, g_{n}\right)$ be a generating sequence of $H$. Then $\left(j\left(g_{1}\right), \ldots, j\left(g_{n}\right)\right)$ is a generating sequence of $G$. Hence, $\left(j\left(g_{1}\right), \ldots, j\left(g_{n}\right)\right)$ lifts via $\pi$ to a generating sequence $\left(m_{1}, \ldots, m_{n}\right)$ of $H(n, G)$. From Proposition 5.1, there exists a surjective homomorphism $f: H \rightarrow H(n, G)$ such that $f\left(g_{i}\right)=m_{i}$ for $i \leq n$. Then $j=\pi \circ f$.
\end{proof}
Proposition 5.2 characterizes $H(n, G)$ in the following sense.

\begin{corollary} Suppose $H$ is a homogeneous group of rank $n$ and there exists a surjective homomorphism $\pi: H \rightarrow G$ such that for any homogeneous group $K$ of rank $n$ and surjective homomorphism $f: K \rightarrow G$ there exists a surjective homomorphism $g: K \rightarrow H$ such that $f=\pi \circ g$. Then $H=$ $H(n, G)$.
\end{corollary}
We will later use Proposition 5.2 in computations to gain information on the structure of homogeneous covers. More precisely, let $G$ be an $n$-generated finite group and let $H$ be a quotient of $G$. Then $H$ is a quotient of $H(n, G)$. Proposition 5.2 then implies that $H(n, H)$ is a quotient of $H(n, G)$.

\begin{proposition} Let $G$ be an n-generated finite group. If $K$ is a quotient of $G$ then $H(n, K)$ is a quotient of $H(n, G)$.
\end{proposition}
\section{Computations of Homogeneous Covers}
In this section, we compute homogeneous covers for some general classes of finite groups along with some more specific examples. The main reason for doing these computations is to learn more about the properties of homogeneous covers. Some of the techniques and results of the previous sections were discovered by making computations. We compute homogeneous covers of arbitrary rank for simple, abelian, nilpotent, and $p q$ groups. Except for simple groups, our general strategy is to use the results from the previous sections to derive information about a given homogeneous cover $H(n, G)$ by relating the structure of $G$ to the structure of $H(n, G)$.

\begin{proposition} Let $G=N \rtimes K$ be an n-generated finite group where $|N|$ and $|K|$ are relatively prime. Then $H(n, G)=N_{0} \rtimes H(n, K)$ for some $N_{0} \leq N^{h_{n}(G)}$.
\end{proposition}
\begin{proof} We view $H(n, G)$ as a subgroup of $G^{h_{n}(G)}$. Define $f: H(n, G) \rightarrow$ $K^{h_{n}(G)}$ by $f\left(g_{1}, \ldots, g_{h_{n}(G)}\right)=\left(\pi\left(g_{1}\right), \ldots, \pi\left(g_{h_{n}(G)}\right)\right)$ where $\pi: G \rightarrow K$ is the quotient map. Note that ker $f=H(n, G) \cap N^{h_{n}(G)}$. By the Schur-Zassenhaus Theorem, $H(n, G)=N_{0} \rtimes H_{0}$ where $N_{0}=\operatorname{ker} f$ and $H_{0}=f(H(n, G))$. Note that $H_{0}$ is an $n$-generated subdirect product of $h_{n}(G)$ copies of $K$. From Proposition 5.4, we know that $H(n, K)$ is a quotient of $H(n, G)$. Hence, $H_{0}=H(n, K)$ since $H(n, K)$ is the largest $n$-generated subdirect product of copies of $K$.
\end{proof}
\begin{corollary} Let $G=N \times K$ be an n-generated finite group where $|N|$ and $|K|$ are relatively prime. Then $H(n, G)=H(n, N) \times H(n, K)$.
\end{corollary}
\begin{proof} Applying the previous proposition, $G=H_{1} \times H(n, K)=H(n, N) \times H_{2}$ where $\left|H_{1}\right|$ and $|K|$ are relatively prime and $\left|H_{2}\right|$ and $|N|$ are relatively prime. It follows that $H(n, G)=H(n, N) \times H(n, K)$.
\end{proof}
\begin{corollary} Let $G=G_{1} \times \cdots \times G_{m}$ be n-generated such that $\left|G_{i}\right|$ and $\left|G_{j}\right|$ are relatively prime for $i \neq j$. Then $H(n, G)=H\left(n, G_{1}\right) \times \cdots \times H\left(n, G_{m}\right)$.
\end{corollary}

We note that Corollary 6.2 does not hold in general for products of groups. As we will show later, $H\left(2, Z_{2}\right)=Z_{2}^{2}$ but $H\left(2, Z_{2}^{2}\right)=Z_{2}^{2}$.

\section{Abelian and Nilpotent Groups}
We now show that homogeneous covers of nilpotent groups and abelian groups have nice forms.

\begin{proposition} Let $G$ be a finite n-generated nilpotent group with distinct Sylow subgroups $P_{1}, \ldots, P_{m}$. Then $H(n, G)=H\left(n, P_{1}\right) \times \cdots \times H\left(n, P_{m}\right)$.
\end{proposition}
\begin{proof} Note that $G=P_{1} \times \cdots \times P_{m}$. Hence, Corollary 6.3 implies that $H(n, G)=H\left(n, P_{1}\right) \times \cdots \times H\left(n, P_{m}\right)$.
\end{proof}
Proposition 7.1 is particularly useful when we know the homogeneous covers of the Sylow subgroups. Homogeneous covers of abelian $p$-groups are easy to compute as we will show. Hence, we will be able to explicitly determine homogeneous covers of finite abelian groups. First, we need a lemma concerning the order of elements in generating sequences for homogeneous groups. We will also need a second lemma which is a consequence of Burnside's Basis Theorem.

\begin{lemma} Let $H$ be a finite homogeneous group of rank $n$. Suppose $x, y \in$ $H$. If $x$ and $y$ are entries of length $n$ generating sequences of $H$ then $|x|=|y|$.
\end{lemma}
\begin{proof} Observe that there exists an automorphism $f: H \rightarrow H$ such that $f(x)=y$.
\end{proof}
\begin{lemma}[Burnside] Let $P$ be a finitely generated abelian p-group. $P=$ $Z_{p^{k_{1}}} \times \cdots \times Z_{p^{k_{m}}}$ for some integers $k_{1}, \ldots, k_{m}>0$. Then $r(P)=m$.
\end{lemma}
\begin{proposition} Let $G$ be an n-generated finite abelian group where $k$ is the largest invariant factor of $G$. Then $H(n, G)=Z_{k}^{n}$.
\end{proposition}
\begin{proof} From the result of Proposition 7.1, it suffices to assume that $G$ is a p-group. $G=Z_{p^{k_{1}}} \times \cdots \times Z_{p^{k_{m}}}$ for some integers $k_{1} \geq \cdots \geq k_{m}>0$. Note that $m \leq n$ since $G$ is $n$-generated. Being a subdirect product of copies of $G$, we must have $H(n, G)=Z_{p^{k_{1}}}^{j_{1}} \times \cdots \times Z_{p^{k_{m}}}^{j_{m}}$ for some integers $j_{1}, \ldots, j_{m} \geq 0$. From Lemma 7.3 and the fact that $r(H(n, G))=n, j_{1}+\cdots+j_{m}=n$.

As $G$ is a quotient of $H(n, G), H(n, G)$ must contain an element of order $p^{k_{1}}$. Note also that any generating sequence of $H(n, G)$ must contain an element of order $p^{k_{1}}$. Lemma 7.2 then implies that any element in a length $n$ generating sequence of $H(n, G)$ must be of order $p^{k_{1}}$. A length $n$ generating sequence $\left(x_{1}, \ldots, x_{n}\right)$ of $H(n, G)$ can be constructed by taking $x_{i} \in Z_{p^{k_{1}}}^{j_{1}} \times$ $\cdots \times Z_{p^{k m}}^{j_{m}}$ to have a generator in the $i t h$ coordinate and identity elements in the other coordinates. It follows that $H(n, G)=Z_{p^{k_{1}}}^{n}$ as required.
\end{proof}
\section{Simple Groups}
The computation of homogeneous covers of finite nonabelian simple groups is made especially easy by applying a result of Hall [4] and [2]. We also note that $r(S) \geq 2$ for any finite nonabelian simple group $S$.

\begin{theorem}[Hall] Let $S$ be a finite nonabelian simple group. Let $s_{1}=$ $\left(s_{11}, \ldots, s_{1 n}\right), \ldots, s_{k}=\left(s_{k 1}, \ldots, s_{k n}\right)$ be generating sequences of $S$. Then $$\left\langle\left(s_{11}, \ldots, s_{k 1}\right), \ldots,\left(t_{1 n}, \ldots, s_{k n}\right)\right\rangle=S^{k}$$ if and only if $s_{i}$ and $s_{j}$ are not equivalent under the action of $\operatorname{Aut}(S)$ for $i \neq j$.
\end{theorem}
\begin{proposition} Let $S$ be a finite nonabelian simple group. For $n \geq$ $2, H(n, S)=S^{h_{n}(S)}$.
\end{proposition}
\begin{proof} The result follows by looking at how we defined $H(n, G)$ at the beginning of the second section.
\end{proof}
\section{Homogeneous Covers of $p q$ Groups}
To conclude this thesis, we compute $H(n, G)$ for a nonabelian $p q$ group $G$ for primes $p$ and $q$ such that $p \mid q-1$. The techniques used in this computation are due to Steve Chase [1]. Some of the results and techniques from previous sections were found during the course of determining $H(n, G)$. Our strategy is somewhat similar to how we computed homogeneous covers of nilpotent groups and abelian groups. We show that $H(n, G)$ has a certain form as a semidirect product as a result of $G$ being a $p q$ group. We then study $n$ generated subdirect products of copies of $G$, which are quotients of $H(n, G)$ by the results of section 2 , to completely determine $H(n, G)$. For the rest of this section, we assume $n \geq 2$ and that $p$ and $q$ are primes such that $p \mid q-1$.

\begin{lemma} Every irreducible $\mathbb{F}_{q}$-representation of $Z_{p}$ is 1-dimensional.
\end{lemma}
\begin{proof} Observe that $\mathbb{F}_{q} Z_{p} \cong \mathbb{F}_{q}[x] /\left(x^{p}-1\right)$ as $\mathbb{F}_{q^{-}}$algebras. $x^{p}-1 \in \mathbb{F}_{q}[x]$ is separable and $\mathbb{F}_{q}$ contains all of the $p t h$ roots since $p \mid q-1$. The Chinese Remainder Theorem then implies that $\mathbb{F}_{q}[x] /\left(x^{p}-1\right) \cong \mathbb{F}_{q}^{p}$ as rings. The result then follows from basic representation theory.
\end{proof}
\begin{lemma} Suppose $G=V \rtimes Z_{p}$ where $V$ is an elementary abelian $q$ group such that the trivial $\mathbb{F}_{q}$-representation of $Z_{p}$ does not appear in the decomposition of $V$ as a sum of irreducible $\mathbb{F}_{q}$-modules with respect to the action of $Z_{p}$ on $V$ in $G$. We write $V=\bigoplus \mathbb{F}_{q}^{+}$additively where $\mathbb{F}_{q}^{+}=\mathbb{F}_{q}$ viewing $\mathbb{F}_{q}$ as an abelian group with respect to addition. Further assume that only one isomorphism class of irreducible $\mathbb{F}_{q}$-representations occurs in the $\mathbb{F}_{q} Z_{p}$ decomposition of $V$. Then $G$ can be generated by $n$ elements if and only if $V$ can be generated by $n-1$ elements.
\end{lemma}
\begin{proof} Viewing elements of $G$ as ordered pairs of elements from $V$ and $Z_{p}$, it is clear that $G$ can be generated by $n$ elements if $V$ can be generated by $n-1$ elements. Conversely, suppose $G$ can be generated by $n$ elements.
$G=\left\langle g_{1}, \ldots, g_{n}\right\rangle$ for some $g_{i}=\left(v_{i}, w_{i}\right)$ where $v_{i} \in V$ and $w_{i} \in Z_{p}$ for $i \leq n$. We may assume that $\left\langle w_{1}\right\rangle=Z_{p}$ after reordering the $g_{i}$ if necessary. As the trivial $\mathbb{F}_{q}$-representation of $Z_{p}$ does not appear in the $\mathbb{F}_{q} Z_{p}$ decomposition of $V, w_{1}(v)=v$ for $v \in V$ if and only if $v=0$. Hence, $\left(i d-w_{1}\right) V=V$ where $i d: V \rightarrow V$ is the identity map. Then $v_{1}=u-w_{1}(u)$ for some $u \in V$. We then have

$$
(u, 1)\left(0, w_{1}\right)(u, 1)^{-1}=\left(v_{1}, w_{1}\right)
$$

As $(u, 1)^{-1} G(u, 1)=G,\left(h_{1}, \ldots, h_{n}\right)$ is a generating sequence of $G$ where $h_{1}=\left(0, w_{1}\right)$ and $h_{i}=(u, 1)^{-1} g_{2}(u, 1)$ for $2 \leq i \leq n$. For any integers $l_{2}, \ldots, l_{n}, G=\left\langle h_{1}, \ldots, h_{n}\right\rangle=G=\left\langle h_{1}, h_{1}^{l_{2}} h_{2}, \ldots, h_{1}^{l_{n}} h_{n}\right\rangle$. As $\left\langle w_{1}\right\rangle=$ $Z_{p}$, it follows that there exists a generating sequence of $G$ of the form $\left(\left(0, w_{1}\right),\left(k_{2}, 1\right), \ldots,\left(k_{n}, 1\right)\right)$.

Finally, we prove that as a group $V$ can be generated by $n-1$ elements. From our assumption on the decomposition of $V$ as an $\mathbb{F}_{q} Z_{p}$-module, $Z_{p}$ acts on $V$ by the multiplication of scalars. Hence, $\langle v\rangle \leq V$ is $Z_{p}$-invariant for all $v \in V$. It follows that $V$ can be generated by $n-1$ elements.
\end{proof}
We now have the tools needed to compute $H(n, G)$.

\begin{proposition} $H(n, G)=V \rtimes Z_{p}^{n}$ where $V=\bigoplus_{i=1}^{(n-1)\left(p^{n}-1\right)} \mathbb{F}_{q}^{+}$. The decomposition of $V$ as a $\mathbb{F}_{q} Z_{p}^{n}$-module consists of the $p^{n}-1$ nontrivial 1-dimensional $\mathbb{F}_{q}$-representations of $Z_{p}^{n}$, and each representation occurs $n-1$ times.
\end{proposition}
\begin{proof} From Propositions 6.1 and 7.4, $H(n, G)=V \rtimes Z_{p}^{n}$ where $V$ is an elementary abelian $q$-group. Note that $H(n, G) \leq G^{h_{n}(G)}$. The Sylow $q$ subgroup $\bigoplus_{i=1}^{h_{n}(G)} \mathbb{F}_{q}^{+}$of $G^{h_{n}(G)}$ decomposes as a sum of nontrivial 1-dimensional $\mathbb{F}_{q} Z_{p}^{h_{n}(G)}$-submodules. Note that the decomposition of $V$ into $\mathbb{F}_{q} Z_{p}^{n}$-submodules contains the trivial $\mathbb{F}_{q}$-representation of $Z_{p}^{n}$ if and only if there exists a nonzero $x \in V$ in the center of $H(n, G)$, which is impossible by the remark of the previous sentence and the fact that $H(n, G) \leq G^{h_{n}(G)}$ is a subdirect product. Hence, $V$ decomposes into nontrivial $\mathbb{F}_{q} Z_{p}^{n}$-submodules.

$V$ must have at least one irreducible $\mathbb{F}_{q} Z_{p}^{n}$-submodule $V_{0} . V_{0}$ corresponds to a nontrivial homomorphism $f_{1}: Z_{p}^{n} \rightarrow \mathbb{F}_{q}^{*}$. Let $f_{2}: Z_{p}^{n} \rightarrow \mathbb{F}_{q}^{*}$ be another nontrivial homomorphism. Note that $f_{1}$ and $f_{2}$ map onto the same cyclic subgroup of $\mathbb{F}_{q}^{*}$. Let $\left(g_{1}, \ldots, g_{n}\right)$ be a generating sequence of $Z_{p}^{n}$. Then $\left(f_{2}\left(g_{1}\right), \ldots, f_{2}\left(g_{n}\right)\right)$ is a generating sequence of $f_{2}\left(Z_{p}^{n}\right) \leq \mathbb{F}_{q}^{*}$. From Proposition 5.1 and its proof, there exists an automorphism $h: H(n, G) \rightarrow H(n, G)$ such that $f_{1}=f_{2} \circ h$.

For $v \in V$ and $z \in Z_{p}^{n}$ we let $z v$ denote $v$ multiplied by the scalar corresponding to $z$ from the homomorphism $Z_{p}^{n} \rightarrow \operatorname{Aut}(V)$. Now $h\left(V_{0}\right)$ is an irreducible $\mathbb{F}_{q} Z_{p}^{n}$-submodule of $V$. Let $v \in V_{0}$ be nonzero. For $w \in Z_{p}^{n}$, wh(v)= $h\left(h^{-1}(w h(v))\right)=h\left(h^{-1}(w) v\right)=h\left(f_{1}\left(h^{-1}(w)\right) \cdot v\right)=h\left(f_{2}(w) \cdot v\right)=f_{2}(w) \cdot h(v)$ where the last equality follows from the fact that scalar multiplication in $V$ commutes with the automorphism $h$. Hence, $f_{2}$ is the homomorphism corresponding to the irreducible submodule $h\left(V_{0}\right)$. We have also shown that $h$ maps any irreducible submodule corresponding to $f_{1}$ to a submodule corresponding to $f_{2}$. Hence, irreducible submodules of $V$ corresponding to $f_{1}$ and $f_{2}$ occur with the same multiplicity $m$.

We have thus shown that the decomposition of $V$ into irreducible $\mathbb{F}_{q} Z_{p}^{n}$ submodules contains $p^{n}-1$ distinct irreducible submodules each occurring with the same multiplicity $m$ for some $m>0$. Let $V_{0}$ be an irreducible $\mathbb{F}_{q} Z_{p}^{n}$ submodule as in the previous paragraphs with corresponding homomorphism $f_{1}: Z_{p}^{n} \rightarrow \mathbb{F}_{q}^{*}$. Factoring out the other irreducible representations of $Z_{p}^{n}$, we see that $\bigoplus_{i=1}^{m} V_{0} \rtimes Z_{p}^{n}$ is a quotient of $H(n, G)$. Note that the kernel of $f_{1}$ is of order $p^{n-1}$. Hence, by factoring out the kernel of $f_{1}$, we see that $\bigoplus_{i=1}^{m} V_{0} \rtimes Z_{p}$ is a quotient of $H(n, G)$ that satisfies the conditions of Lemma 9.2. As $H(n, G)$ is $n$-generated, it follows that $\bigoplus_{i=1}^{m} V_{0}$ is $n$-1-generated. Hence, $m \leq n-1$.

We must now show that $m=n-1$. Note that $H=\bigoplus_{i=1}^{n-1} V_{0} \rtimes Z_{p}$, where the action of $Z_{p}$ on $V_{0}$ is defined by $f_{1}$, is a $n$-generated subdirect product of $n-1$ copies of $G$. To see this, let $x$ be a generator of the Sylow $q$-subgroup of $G$ and let $y$ be a generator of the Sylow $p$-subgroup of $G$. Then $H \leq G^{n-1}$ is the group generated by $(y, y, \ldots, y) \in G^{n-1}$ and the $n-1$ elements $(x, 1, \ldots, 1),(1, x, \ldots, 1), \ldots,(1,1, \ldots, x) \in G^{n-1}$ where $1 \in G$ is the identity element. From Proposition 4.1, $H$ is a quotient of $H(n, G)$. Hence, $H$ is a subgroup of $H(n, G)$ since the quotient map $H(n, G) \rightarrow H$ splits. Thus, $m=n-1$ as required.
\end{proof}
\section*{Acknowledgements}
I would like to thank Keith Dennis and Steve Chase for being sources and inspiration for the work in this thesis. I would also like to thank Dan Collins for writing excellent notes on group theory.

\section*{References}

\, \, \, [1] Chase, Steve. Homogeneous Covers of pq-Groups, 2007.

[2] Collins, Dan. Generating Sequences of Finite Groups, 2013.

[3] The GAP Group, GAP - Groups, Algorithms, and Programming, Version 4.7.4; 2014, (\href{http://www.gap-system.org}{http://www.gap-system.org}).

[4] Gaschütz, Wolfgang. Zu einem von B. H. und H. Neumann gestellten Problem, Math. Nachr, 14 (1955), 249-252

[5] Hall, Philip. The Eulerian Functions of a Finite Group, Quarterly Journal of Mathematics, 7 (1936), 134-151

\end{document}